\documentclass[12pt, final]{l4dc2023}

\title[Discrete optimization via oscillations]{A Dynamical Systems Perspective on Discrete Optimization}
\usepackage{times}
\usepackage{bbm}
\usepackage{tikz}
\usepackage{dirtytalk}
\usepackage{graphicx}
\usepackage{pgfplots}
\usepackage{amsmath,amssymb,amsfonts}
\graphicspath{ {./figures/} } 

\author{%
 \Name{Tong Guanchun} \Email{guanchuntong@gmail.com}\\
 \addr Max Planck Institute for Intelligent Systems\\
72076 T\"{u}bingen, Germany
 \AND
 \Name{Michael Muehlebach} \Email{michael.muehlebach@tuebingen.mpg.de}\\
 \addr Max Planck Institute for Intelligent Systems\\
72076 T\"{u}bingen, Germany%
}

\begin{document}

\maketitle

\begin{abstract}%
We discuss a dynamical systems perspective on discrete optimization. Departing from the fact that many combinatorial optimization problems can be reformulated as finding low energy spin configurations in corresponding Ising models, we derive a penalized rank-two relaxation of the Ising formulation. It turns out that the associated gradient flow dynamics exactly correspond to a type of hardware solvers termed oscillator-based Ising machines. We also analyze the advantage of adding angle penalties by leveraging random rounding techniques. Therefore, our work contributes to a rigorous understanding of oscillator-based Ising machines by drawing connections to the penalty method in constrained optimization and providing a rationale for the introduction of sub-harmonic injection locking. Furthermore, we characterize a class of coupling functions between oscillators, which ensures convergence to discrete solutions. This class of coupling functions avoids explicit penalty terms or rounding schemes, which are prevalent in other formulations.
\end{abstract}


\section{Introduction}


Dynamical systems and optimization algorithms are intimately related. Many physical systems, which are described by the language of dynamical systems and differential equations, invariably seek optimal configurations or trajectories that are stationary. This is known as Hamilton's principle of stationary action, which remains central in modern physics and mathematics. Similarly, optimization algorithms are iterative computational procedures, which can be naturally viewed as the motion of a physical system that gravitates towards low energy solutions.

Indeed, there is sustained interest in exploring the relation between dynamical systems and optimization. Gradient flow dynamics has been a prominent topic in this regard \citep{Brockett1988DynamicalST,Bloch1992OnTG,Helmke1994OptimizationAD,Absil2004ContinuousDS}. Over the last decade, gradient-based methods overshadow second-order methods in solving practical optimization problems involving large data-sets and over-parameterized neural networks. This has led to a vibrant research activity investigating such modern-era optimization, with dynamical systems offering a unique viewpoint to, for example, understand the acceleration phenomenon \citep{Nesterov2018LecturesOC,Su2016ADE, Wibisono16Variational, BinHu17Diss, Muehlebach2019ADS, Muehlebach2021OptimizationWM}. We are also aware of related perspectives involving saddle-point dynamics \citep{ArrowHurwicz58, Cherukuri2015SaddlePointDC, Corts2016DistributedCF, Qu2018OnTE}, feedback and adaptive control \citep{Ariyur2003RealTimeOB, Belgioioso2022OnlineFE}, and the interplay of circuit theory and monotone operators \citep{Chaffey2021MonotoneRC}. 

While the analysis of optimization algorithms from a dynamical system perspective yields many insights and gives new algorithms with favorable properties, the majority of the above work focuses on continuous optimization problems. In discrete optimization, the analogies to dynamical systems are far from obvious due to the combinatorial structure of the problem. Should the state space be discrete or continuous? Should the dynamics be discrete or continuous? No convincing answers seem to be attainable at first glance.

However, combinatorial optimization problems are ubiquitous and represent bottlenecks in many applications. For instance, with the advent of gene editing methods like CRISPR \citep{Jinek2012APD} and high-throughput sequencing technology, scientists are able to acquire an unprecedented wealth of interventional data, which has the potential for improving our understanding of gene regulatory networks and developing effective therapeutics. However, the design of effective gene knockout and knockdown experiments, which are essential for elucidating genotype-phenotype relationships, remain prohibitive due to the astronomical figure of all possible combinations of genetic interventions. Although this attracted researchers investigating experimental design for causal inference and system identification \citep{Hauser2011CharacterizationAG,AlbertoGiovanni, Ghassami2017BudgetedED, Agrawal2019ABCDStrategyBE}, the efficient solution to combinatorial problems is a key factor that prevents automatic causal discovery.

The interplays between dynamical systems and combinatorial optimization often come under the rubric of \textit{analog computing}.  In recent years, researchers found that many hard combinatorial optimization problems can be reformulated as finding low energy spin configurations of corresponding Ising models \citep{Lucas2014IsingFO}, which led to the development of hardware solvers \citep{Mohseni2022IsingMA}. These so-called post-von Neumann Ising machines have been reported to solve large-scale combinatorial optimization problems at fast speeds, while requiring little energy. Equally intriguing are the theoretical implications of these special-purpose computing architectures for our understanding of non-convex optimization problems. Even numerical simulations of Ising machines running on a conventional digital computer constitute competitive new heuristics for hard problems. By lifting the notion of  \say{algorithms} into the operational dynamics of Ising machines, we gain opportunities to apply tools from systems and control theory, statistical mechanics, and the theory of polynomials to study combinatorial optimization. 

In this paper, we restrict our scope to a type of Ising hardware solvers based on networks of coupled self-sustaining oscillators, termed \textit{oscillator-based Ising machines}, as proposed by \cite{Wang2019OIMOI}. The main idea is that the phase dynamics of coupled oscillators under the influence of sub-harmonic injection locking admits a Lyapunov function that is closely related to the Ising Hamiltonian of the coupling graph. This allows for approximations to a certain class of combinatorial optimization problems. Practically, the oscillator-based Ising machine is appealing for being implementable with standard complementary metal–oxide–semiconductor technology with the prospect of miniaturization and mass production. Theoretically, this scheme only utilizes the physical mechanism of coupled nonlinear oscillators while other approaches typically go beyond classical dynamics, for instance, by harnessing quantum phenomena \citep{Yamamoto2020CoherentIM}.

A peculiar feature of oscillator-based Ising machines is the introduction of the above mentioned sub-harmonic injection locking. Consider an autonomous oscillator with a natural frequency of $f_0$ that is perturbed by a small periodic input signal with a nearby frequency of $f_1 \approx f_0$. Injection locking describes the phenomenon when the phase response of the oscillator becomes entrained with the input signal. Thus, both the frequency and phase of the oscillator are locked to the input. Mutual injection locking is the underlying mechanism of synchronization phenomena in complex networks \citep{Drfler2014SynchronizationIC}. It is also possible for the external perturbation to have a frequency of $2f_1$, i.e., about twice as fast as the oscillator. Then, the oscillator can still be locked at the frequency $f_1$, which is the sub-harmonic of the perturbation. However, the phase response can now have two possible configurations: either it settles at $0$ (in phase) or at $\pi$ (anti-phase). It is this binarization mechanism that enables the implementation of Ising spins with the phase logic of oscillators. Throughout this article, when phases or angles of oscillators settle down at binarized states, we call this \textit{bi-stability}.

Bi-stability is also a recurring theme for the main contributions of this paper. We derive the dynamics of oscillator-based Ising machines using penalty methods from constrained optimization. This provides a solid theoretical foundation for the introduction of sub-harmonic injection locking and also enables extensions and improvements, for example by applying augmented Lagrangian approaches. In the context of max-cut problems, rank-two relaxation \citep{Burer2002RankTwoRH} is one of the most performant heuristics, and oscillator-based Ising machines, which can be viewed as rank-two relaxation with additional penalties, enjoy even superior performance \citep{Wang2019OIMOI}. We analyze this benefit of introducing penalty that promotes bi-stability from a random rounding perspective \citep{Goemans1995ImprovedAA}. It turns out that the shrinkage of the range of all possible pairwise angle differences leads to better lower bounds for the expected edges, which is directly related to the suboptimality or approximation ratio of the obtained solution. We further show that bi-stability can be achieved without explicit penalty terms by introducing a class of generalized coupling functions. The contributions of our paper are also summerized in Fig.~\ref{main-concepts}.

\begin{figure}[!h]
\label{main-concepts}
\centering

\begin{tikzpicture}

    \node at (5.5,2) (shil) {SHIL};
    \node at (1.5,0) (constraint) {angle constraint};
    \node at (1.5, 2) (penalty) {angle penalty};
    \node at (9, 2) (g-coupling) {generalized coupling};
    \node at (5.5, 0) (bistability) {bi-stability};
    \node at (12, 0) (rounding) {random rounding};
     \draw (constraint) -- (bistability) node [midway, above, sloped] (TextNode) {Sec.~\ref{angle-constraints}};
     \draw (constraint) -- (penalty) node [midway, above, sloped] (TextNode) {Sec.~\ref{bistability via penalty}};
    \draw (shil) -- (penalty) node [midway, above] (TextNode) {Sec.~\ref{bistability via penalty}};
    \draw (g-coupling) -- (bistability) node [midway, above,sloped] (TextNode) {Sec.~\ref{g-coupling}};
    \draw (g-coupling) -- (rounding) node [midway, above,sloped] (TextNode) {Sec.~\ref{g-coupling}};
    \draw (g-coupling) -- (bistability) node [midway,below,sloped] (TextNode) {};
    \draw (bistability) -- (rounding) node [midway, above,sloped] (TextNode) {Sec.~\ref{bistability-rounding}};
    \draw (g-coupling) -- (rounding) node [midway, above,sloped] (TextNode) {};
    \draw (shil) -- (bistability) node [midway, above, sloped] (TextNode) {OIM};
    

    
    \draw [<->] (shil) -- (penalty);
    \draw [->] (constraint) -- (bistability);
    \draw [->] (shil) -- (bistability);
    \draw [->] (constraint) -- (penalty);
    \draw [->] (g-coupling) -- (bistability);
    \draw [-] (g-coupling) -- (rounding);
    \draw [-] (bistability) -- (rounding);

\end{tikzpicture}
\caption{The figure illustrates the relationship among the main concepts in this article, where SHIL stands for sub-harmonic injection locking and OIM for oscillator-based Ising machines.}
\end{figure}
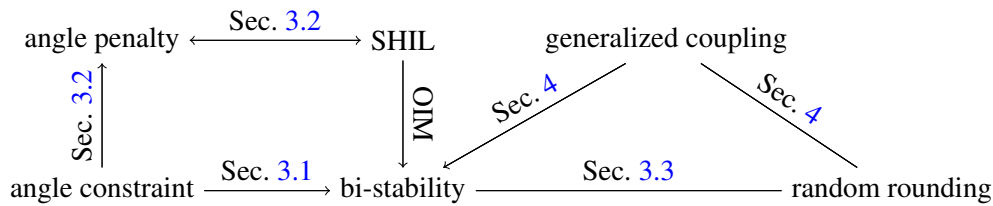

\section{Ising formulation of combinatorial optimization on the example of max-cut}

\subsection{Ising model and its applications}

The Ising model was initially conceived by Wilhelm Lenz for the study of the para/ferromagnetic phase transition and named after his student, Ernst Ising, who analyzed the one-dimensional version in his dissertation. Lenz followed the contemporary view that a magnet is made of elementary pieces that behave themselves as small magnets. He further made the assumption that the small magnets are only allowed to point in two opposite directions  \citep{LenzBeitragZV}. Ising further improved the model and included interactions among these elementary magnets \citep{IsingBeitragZT}.

More formally, we consider an undirected graph $G=(V, E)$ where $V=\{1, \ldots, n\}$ indexes the positions of the elementary magnets. Each elementary magnet $i \in V$ may take the direction $\sigma_i \in\{-1,+1\}$, called the \textit{spin} at $i$. The collection $\left(\sigma_i: i \in V\right) \in\{-1,1\}^V$ of all spins is called the \textit{spin configuration}, and models the state of the entire magnet. When there is no external magnetic field, the \textit{Hamiltonian} of the system is given by 

\begin{equation}
    \label{ising-hamiltonian}
    H(\sigma;G, J):=-\sum_{\{i, j\} \in E} J_{ij} \sigma_i \sigma_j,
\end{equation}
where for any two adjacent nodes $i, j \in V$, there is a coupling constant $J_{ij}$.

The Ising model did not develop rapidly after its introduction. In fact, Ising himself was only aware of one citation of his paper and abandoned academia after publishing the paper. Nonetheless, after over 100 years of evolution \citep{DuminilCopin2022100YO}, the model is now the cornerstone of statistical mechanics and a versatile mathematical tool for the study of cooperative phenomena in complex systems. Over the years, the Ising model has found many applications, including statistics \citep{GraphicalLauritzen,Wainwright2008GraphicalME}, neural computation \citep{Amari1972LearningPA,Hopfield1982NeuralNA}, and complex systems \citep{Harken-1983-advanced,Kuramoto1984ChemicalOW, Parisi2023NobelLM}.

This article uses the Ising model as a point of departure, due to the fact that many hard combinatorial optimization problems can be reformulated as Ising problems, i.e., finding ground states of the corresponding Ising Hamiltonian. This includes
all of Karp’s 21 NP-complete problems \citep{Karp1972}, where we will elaborate on the example of max-cut in the next section.



\subsection{Max-cut problem and its Ising formulation}

For simplicity of exposition, we focus on unweighted graphs throughout this article.
\begin{definition}
\normalfont Let $G=(V, E)$ be an unweighted undirected graph, where $V=\{1,\ldots, n\}$ denotes the vertices and $E$ the edges. A \textit{cut } of $G$ is defined by a partition $(V_1, V_2)$ of $V$, where $V_1$ is a non-empty subset of $V$, $V_1 \cup V_2=V$ and $V_1 \cap V_2=\emptyset$. The \textit{maximum cut problem} is to find a cut with a maximum number of edges connecting $V_1$ to $V_2$.
\end{definition}

Max-cut is NP-hard and serves as a canonical problem in combinatorial optimization and theoretical computer science. The reason for its significance is two-fold: on the one hand it arises from important practical applications such as circuit design, and on the other hand, the techniques originally introduced for max-cut are often adapted and repurposed for broader classes of algorithms. For instance, the seminal paper by \cite{Goemans1995ImprovedAA} gave an approximation ratio $\alpha\approx 0.878$ for max-cut. If the unique game conjecture is true \citep{Trevisan2012OnKU}, this is the best possible approximation ratio. Moreover, the introduced solution strategy, i.e., semidefinite programming and random rounding, became paradigms for the design and analysis of approximation algorithms for computationally hard problems far beyond max-cut \citep{Grtner2012ApproximationAA}. The random rounding is also instrumental for ideas presented herein, in particular for the coupling functions that achieve bi-stability. We conclude the section by demonstrating how max-cut can be formulated as minimizing an Ising Hamiltonian.

\begin{proposition}
\normalfont Given an unweighted undirected graph $G=(V, E)$ with $V=\{1, \ldots, n\}$ and adjacency matrix $A$. Finding the max cut in $G$ is equivalent to solving the following Ising problem:

\begin{equation}
\label{eq:ising-maxcut}
\begin{array}{ccc}
&\operatorname{min}  \sum\limits_{1 \leqslant i<j \leqslant n} a_{i j} \sigma_i \sigma_j  &\operatorname{s.t.}  \quad \sigma_i \in\{-1,+1\}, \quad \forall i \in\{1, \ldots, n\}.
\end{array}
\end{equation}
\end{proposition}

\begin{proof}
\normalfont We select the cut defined by $V_1=\left\{i \mid \sigma_i=+1\right\}$ and $V_2=\left\{i \mid \sigma_i=-1\right\}$. A dichotomy of the vertex set induces a trichotomy of the edge set, i.e., for each $\{i,j\} \in E$, exactly one of the three cases holds: both its incidental vertices are in $V_1$, both of its incidental vertices are in $V_2$, or its incidental vertices are in distinct subsets. We denote the total edge weights in the three cases by $S_1, S_2$, and $S_{\mathrm{cut}}$, respectively.
Then, the objective function can be rearranged accordingly:

\[\begin{aligned}
\sum_{1 \leqslant i<j \leqslant n} a_{i j} \sigma_i \sigma_j&=\frac{1}{2}\sum_{i \in V_1, j \in V_1} a_{i j}(+1)(+1)+\frac{1}{2}\sum_{i \in V_2, j \in V_2} a_{i j}(-1)(-1) \\
&+\frac{1}{2}\sum_{i \in V_1, j \in V_2} a_{i j}(+1)(-1)+\frac{1}{2}\sum_{i \in V_2, j \in V_1 } a_{i j}(-1)(+1)\\
&= S_1 + S_2 - S_{\mathrm{cut}}
= \sum_{1 \leqslant i<j\leqslant n} a_{i j}-2 S_{\text{cut}}.
\end{aligned}\]
The sum $\sum_{1 \leqslant i<j\leqslant n} a_{i j}$ is the total number of edges of the graph, and is therefore independent of the spin configuration. Hence minimizing the objective is equivalent to maximizing $S_{\mathrm{cut}}$. 
\end{proof}
\section{Multiple views of bi-stability: from angle constraints to injection locking}

\subsection{Geometry of the Ising formulation and the role of angle constraints}
\label{angle-constraints}

Although we obtained the Ising reformulation of max-cut in the previous section, the discrete nature of the problem remains inact. In this section, we relax the space of spins from discrete $\left(\sigma_i \in \mathbb{S}^0:=\{-1,+1\}\right)$ to continous $\left(u_i \in \mathbb{S}^1\right)$ and leverage a parametrization in polar coordinates. Finally, we will introduce constraints to enforce bi-stability.




More precisely, let  $u_i \in \mathbb{R}^2$ for each $i \in\{1, \ldots, n\}$ and let $\left\{e_1, e_2\right\} \subseteq \mathbb{R}^2$ be the standard basis, then the Ising formulation can be translated to the following:




\begin{equation}
\label{eq:dot-form}
\operatorname{min}\quad  -\sum\limits_{1 \leq i<j \leq n} J_{i j} u_i^{\top} u_j \quad \operatorname{s.t.} \quad \left\|u_i\right\|_2^2=1,\quad e_2^{\top} u_i=0, \quad \forall i \in\{1, \ldots, n\},
\end{equation}
where $\|\cdot\|_2$ denotes the Euclidean norm. The constraints in \eqref{eq:dot-form} can be interpreted as a set of \textit{\say{length constraints}} and a set of \textit{\say{angle constraints}}. Recall our historical account of the Ising model, where the introduction of angle constraints is reminiscent of Lenz's assumption on formalizing the spins. He challenged the view that elementary magnets can rotate freely within a solid and made an analogy to crystals selecting certain directions corresponding to their
symmetries. Thus, he suggested that elementary magnets align themselves with respect to their neighbours, which corresponds to either pointing in
the same or opposite direction. 

We further represent the vectors in $\mathbb{S}^1$ using polar coordinates: $u_i=\left(\cos \theta_i,\sin \theta_i\right)$, and note that $u_i^{\top} u_j=\cos \left(\theta_i-\theta_j\right)$. The set of angle constraints simplifies to $e_2^{\top} u_i=\sin \theta_i$,
whereby the length constraints are implicitly satisfied with the new representation.
Hence the optimization problem in \eqref{eq:dot-form} now reads:

\begin{equation}
\label{eq:polar}
\operatorname{min} \quad -\sum\limits_{1 \leq i<j \leq n} J_{i j} \operatorname{cos}(\theta_{i} - \theta_{j}) \quad \operatorname{s.t.} \quad
 \sin \theta_i=0, \quad \forall i \in \{1, \ldots, n\}.
\end{equation}

\subsection{Angle penalties and sub-harmonic injection locking}
\label{bistability via penalty}
The constrained problem \eqref{eq:polar} can be cast as an unconstrained optimization problem via the penalty method \citep{Bertsekas1999NonlinearP}. The penalized objective is 
\begin{equation}
\label{eq:penalized objective}
L(\theta ; J, \mu)=-\sum_{1 \leqslant i<j \leqslant n} J_{i j} \cos \left(\theta_i-\theta_j\right)+\frac{\mu}{2} \sum_{i=1}^n \sin ^2 \theta_i.
\end{equation}
We will show in the following that the gradient flow dynamics associated with the above energy function \eqref{eq:penalized objective} operates as an oscillator-based Ising machine, as described by Eq. $9$ in \cite{Wang2019OIMOI}:
\begin{equation}
    \frac{d}{d t} \theta_i(t)=-K  \sum_{j \neq i} J_{i j} \sin \left(\theta_i(t)-\theta_j(t)\right)-K_s  \sin \left(2 \theta_i(t)\right).
\end{equation}

\begin{proposition}
    \normalfont
\label{gradient flow ising machine}
The gradient system $\dot{\theta}=-\nabla L(\theta ; J, \mu)$ is equivalent to the dynamics of an oscillator-based Ising machine.

\end{proposition}

\begin{proof}
For $i = 1, \ldots, n$, we have
\begin{equation}
    \begin{aligned}
\frac{\partial L}{\partial \theta_i} &=-\sum_{j<i} J_{j i} \sin \left(\theta_j-\theta_i\right)+\sum_{j>i} J_{i j} \sin \left(\theta_i-\theta_j\right)+\frac{\mu}{2} \sin 2 \theta_i \\
&=\sum_{j \neq i} J_{i j} \sin \left(\theta_i-\theta_j\right)+\frac{\mu}{2} \sin 2 \theta_i.
\end{aligned}
\end{equation}
Hence we obtain the following equation for gradient flow:
\begin{equation}
    \dot{\theta}_i=-\frac{\partial L}{\partial \theta i}=-\sum_{j \neq i} J_{i j} \sin \left(\theta_i-\theta_j\right)-\frac{\mu}{2} \sin 2 \theta_i.
\end{equation}
By choosing the penalty coefficient such that $K_s/K=\mu/2$,
we recover the dynamics for an oscillator-based Ising machine.
\end{proof}

\begin{remark}
\normalfont
The above equivalent formulation of oscillator-based Ising machines is not simply a tautology. In fact, the sub-harmonic injection locking signal in Ising machines was suggested heuristically, where it is assumed that the coupling will not affect the bi-stability of the individual oscillator, which is enforced by the external input. Our derivation is based on penalty methods, which have a rigorous foundation in constrained optimization. This provides a rationale for the introduction
of sub-harmonic injection locking mechanisms and also enables extensions, for example by applying augmented Lagrangian approaches.
\end{remark}


\subsection{Why adding penalties for bi-stability? A view from random rounding}
\label{bistability-rounding}
Our above result, when specialized to the max-cut problem (with $J_{ij} = -a_{ij}$), can be seen as rank-two relaxation \citep{Burer2002RankTwoRH} with angle penalty. Although the rank-two relaxation (without penalty) is one of the most performant heuristics for max-cut according to a recent benchmark study \citep{Dunning2018WhatWB}, adding penalties is reported to achieve even better results \citep{Wang2019OIMOI}. Next, we show that the classic random rounding scheme \citep{Goemans1995ImprovedAA} can be adapted to analyze the superior performance of the penalized rank-two relaxation for max-cut.

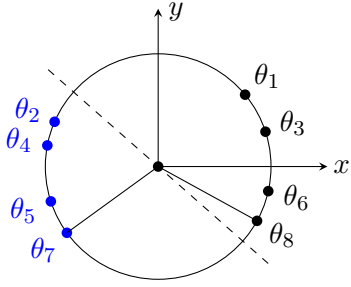
\begin{figure}[h]
 \centering
 \begin{tikzpicture}

  \coordinate (O) at (1,2);
  \def\radius{1.5cm}

  \draw (O) circle[radius=\radius];
  \fill (O) circle[radius=2pt] ;

  \pgfmathsetmacro{\angleA}{0.22*180}
  \path (O) ++(\angleA:\radius) coordinate (A);
  \pgfmathsetmacro{\angleB}{0.1*180}
  \path (O) ++(\angleB:\radius) coordinate (B);
  \pgfmathsetmacro{\angleC}{-0.07*180}
  \path (O) ++(\angleC:\radius) coordinate (C);
  \pgfmathsetmacro{\angleD}{-0.16*180}
  \path (O) ++(\angleD:\radius) coordinate (D);

   \pgfmathsetmacro{\angleE}{0.87*180}
  \path (O) ++(\angleE:\radius) coordinate (E);
  \pgfmathsetmacro{\angleF}{0.94*180}
  \path (O) ++(\angleF:\radius) coordinate (F);
  \pgfmathsetmacro{\angleG}{1.1*180}
  \path (O) ++(\angleG:\radius) coordinate (G);
  \pgfmathsetmacro{\angleH}{1.2*180}
  \path (O) ++(\angleH:\radius) coordinate (H);

    \pgfmathsetmacro{\angleI}{0.77*180}
   \path (O) ++(\angleI:1.3*\radius) coordinate (I);

    \pgfmathsetmacro{\angleJ}{1.77*180}
   \path (O) ++(\angleJ:1.3*\radius) coordinate (J);

       \pgfmathsetmacro{\angleX}{0}
   \path (O) ++(\angleX:1.5*\radius) coordinate (X);

          \pgfmathsetmacro{\angleY}{90}
   \path (O) ++(\angleY:1.4*\radius) coordinate (Y);

   \fill[black] (A) circle[radius=2pt] ++(\angleA:1em) node {$\theta_1$};
  \fill[black] (B) circle[radius=2pt] ++(\angleB:1em) node {$\theta_3$};
  \fill[black] (C) circle[radius=2pt] ++(\angleC:1em) node {$\theta_6$};
  \fill[black] (D) circle[radius=2pt] ++(\angleD:1em) node {$\theta_8$};
  
  \fill[blue] (E) circle[radius=2pt] ++(\angleE:1em) node {$\theta_2$};
  \fill[blue] (F) circle[radius=2pt] ++(\angleF:1em) node {$\theta_4$};
  \fill[blue] (G) circle[radius=2pt] ++(\angleG:1em) node {$\theta_5$};
  \fill[blue] (H) circle[radius=2pt] ++(\angleH:1em) node {$\theta_7$};

   \draw (O) -- (D);
    \draw (O) -- (H);
     \draw [-stealth](O) -- (X);
     \draw [-stealth](O) -- (Y);

      \fill[black] (X) circle[radius=0pt] ++(\angleX:0.5em) node {$x$};
      \fill[black] (Y) circle[radius=0pt] ++(\angleY:-0.1em) [anchor=west]node {$y$};
    \draw [dashed] (I) -- (J);
    
\end{tikzpicture}

\caption{The figure illustrates how random rounding works in polar coordinates: a given angle configuration is partitioned into two subsets by choosing a random line (dashed) through the origin (the normal vector is uniformly distributed on the unit circle).}
\label{fig: random rounding}
\end{figure}
Given an angle configuration $\theta=\left(\theta_1, \ldots, \theta_n\right) \in \mathbb{T}^n:=\mathbb{S}^1 \times \cdots \times \mathbb{S}^1$, the random rounding procedure gives a partition of $\theta$ by selecting a random line and assigning particles that are in different half spaces into distinct subsets $V_1$ and $V_2$, which is illustrated in Fig.~\ref{fig: random rounding}. 
Let $\mathbb{E}W_\theta$ be the expected number of edges that connect $V_1$ and $V_2$, when applying random rounding to the angle configuration $\theta$. By linearity of expectation and observing that the probability of the random line separating two points is proportional to the angle between the two points, we have
\begin{equation}
\label{eq:random-rounding expectation}
  \mathbb{E} W_\theta=\sum_{i<j} a_{i j} \mathbb{P}\left(\theta_i, \theta_j \text { in different half-spaces}\right)=\sum_{i<j} a_{i j} \frac{\left|\theta_i-\theta_j\right|_{\mathbb{S}^1}}{\pi}, 
\end{equation}
where $|\cdot|_{\mathbb{S}^1}$ denotes the shortest distance on $\mathbb{S}^1$. In the following, we are interested in relating $\mathbb{E} W_\theta$ to the max-cut value $W_{\mathrm{mc}}$, which is established by the \textit{approximation ratio}\footnote{As we will discuss later, $\alpha$ only relates to $W_{\mathrm{mc}}$ if the configuration $\theta$ is a minimizer of (\ref{eq:penalized objective}).}:
\begin{equation}
\label{eq: ratio}
   \alpha:=\min _\theta \frac{\left|\theta_i-\theta_j\right|_{\mathbb{S}^1} / \pi}{\left(1-\cos \left(\theta_i-\theta_j\right)\right) / 2}= \min _{x \in[0, \pi]} \frac{x / \pi}{(1-\cos x) / 2}.
\end{equation}
The reason that we can restrict our domain to $[0, \pi]$ is that both $|\cdot|_{\mathbb{S}^1}$ and $\cos (\cdot)$ are even functions and are symmetric about $\pi$. They are also depicted by the red line and blue curve in the right sub-plot in Fig.~\ref{fig: energy coupling and lowerbound ratio}. We use the words \textit{approximation ratio} and \textit{lower bound ratio} interchangeably in this article. The lower bound ratio can be improved due to the introduction of the angle penalty term, which forces the phases of the oscillators to cluster towards $0$ and $\pi$, where the ratio evaluates to $1$. We formally state this as follows (the proof of which is standard and follows the classic random rounding analysis):

\begin{proposition}
\label{thm:I_mu}
\normalfont
Let $\underline{\theta} \in \mathbb{T}^n$ be an angle configuration that minimizes \eqref{eq:penalized objective} with $J_{ij} = -a_{ij}$. Then, the expected number of edges for a random rounding partition satisfies

\begin{equation}
    \mathbb{E} W_{\underline{\theta}} \geqslant\left(\min _{x \in I_\mu} \frac{2}{\pi} \frac{x}{1-\cos (x)}\right)W_{\mathrm{mc}},
\end{equation}
where $I_{\mu} \subseteq [0,\pi]$ is the shrinked range of all possible pairwise angle differences (due to the angle penalty) and $W_{\mathrm{mc}}$ is the value of max-cut. 




\end{proposition}
\section{How to achieve bi-stability without penalty?}
\label{g-coupling}

As discussed above, bi-stability is a desirable property, which is reflected in the lower bound on the expected number of edges that arises from a random rounding procedure. Now we ask a more ambitious question: how can we retain the favored bi-stability while getting rid of the penalty term? It turns out that we can achieve this by defining a generalized coupling function as follows \citep[also considered by][]{Steinerberger2021MaxCutVK}:

\begin{equation}
\label{eq:energy-g}
    L(\theta;A,g)=\sum_{i<j} a_{i j} g\left(\theta_i-\theta_j\right),
\end{equation}
where $g: \mathbb{S}^1 \rightarrow[-1,1]$ is assumed (i) to be differentiable, (ii) to attain its maximum value of $1$ at $0$ and its minimum value of $-1$ at $\pi$, and (iii) is an even function. We denote the class of functions that satisfy these requirements by \textbf{$\mathcal{G}$}. For $g \in \mathcal{G}$, we conclude that  $g$ is symmetric with respect to $x=\pi$, namely, $g(x) = g(-x)= g(2\pi-x)$. Hence when there exists an edge between nodes $i$ and $j$, these properties promote that the corresponding angles $\theta_i$ and $\theta_j$ have a difference of $\pi$, so as to contribute the negative term $-a_{i j}$ to the total energy.

\begin{figure}[h]

  \centering

  \begin{tikzpicture}
   
  \draw[->] (0, 0) -- (6.6, 0) node[right] {};
  \draw[->] (0, -1.2) -- (0, 1.3) node[above] {};

\draw[dashed] (0,-1) -- (pi,-1);
  \draw[dashed] (2*pi,0) -- (2*pi,1);
   \filldraw[black] (pi,0) circle (0.018pt) node[anchor=north]{$\pi$};
    \filldraw[black] (2*pi,0) circle (0.018pt) node[anchor=north]{$2\pi$};
\filldraw[black] (0,1) circle (0.018pt) node[anchor=east]{$1$};
\filldraw[black] (0,-1) circle (0.018pt) node[anchor=east]{$-1$};
\filldraw[black] (0,0) circle (0.018pt) node[anchor=east]{$0$};

  \draw[scale=1, domain=0:2*pi, smooth, variable=\x, blue] plot ({\x}, {cos(deg(\x))});
  \draw[scale=1, domain=0:pi, smooth, variable=\x, black] plot ({\x}, {1-2*\x^2/pi^2});
  \draw[scale=1, domain=pi:2*pi, smooth, variable=\x, black] plot ({\x}, {1-2*(2*pi-\x)^2/pi^2});

\end{tikzpicture}
\begin{tikzpicture}
  \draw[->] (0, 0) -- (6.8, 0) node[right] {};
  \draw[->] (0, 0) -- (0, 2.36) node[above] {};
     \draw[dashed] (0,2) -- (2*pi,2);
  \draw[dashed] (2*pi,2) -- (2*pi,0);
   \filldraw[black] (2*pi,0) circle (0.018pt) node[anchor=north]{$\pi$};
\filldraw[black] (0,2) circle (0.018pt) node[anchor=east]{$1$};
\filldraw[black] (0,0) circle (0.018pt) node[anchor=north]{$0$};
  \draw[scale=2, domain=0:pi, smooth, variable=\x, blue] plot ({\x}, {(1-cos(deg(\x)))/2});
  \draw[scale=2, domain=0:pi, smooth, variable=\x, red] plot ({\x}, {\x/pi});
   \draw[scale=2, domain=0:pi, smooth, variable=\x, black] plot ({\x}, {\x^2/pi^2});
  
\end{tikzpicture}

\caption{The left plot shows the cosine coupling function in blue, and in black a generalized coupling $g(x)=1-2 x^2 / \pi^2$ for $x \in[0, \pi]$, which, as we shall see below, promotes bi-stability and leads to an optimal approximation ratio. The right plot shows $(1-\cos (x)) / 2$ and $(1-g(x)) / 2$, which represents the denominator that defines the approximation ratio, with corresponding colors. The red line on the right illustrates $x / \pi$, which corresponds to the numerator.}
\label{fig: energy coupling and lowerbound ratio}

\end{figure}
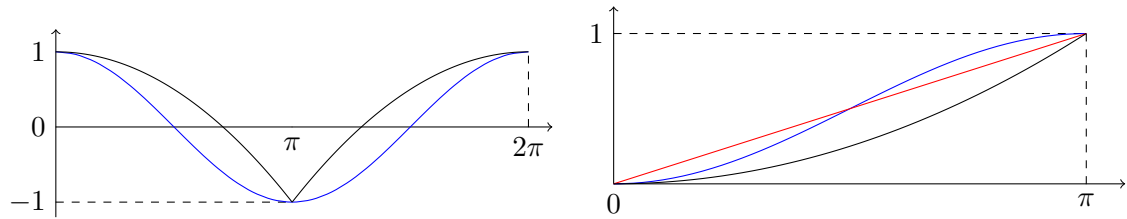

We proceed by noting that Prop.~\ref{thm:I_mu} ensures that bi-stable minimal energy configurations imply an optimal ratio for random rounding. We now ask the converse: does a coupling function with an a-priori optimal lower bound ratio lead to bi-stability? The answer is yes provided that a unique max-cut solution exists. For Erd\H{o}s R\'enyi random graphs, the existence of a unique max-cut solution is related to the number of neighbors that each vertex is connected to \citep{Ling2019OnTL}.

\begin{theorem}
\label{thm:optimal coupling}
\normalfont
Let $G=(V, E)$ be an unweighted, undirected graph and $A \in\{0,1\}^{n \times n}$ its adjacency matrix. Assume $G$ has a unique pair of max-cut solutions given by $\{V_1^*, V_2^*\}$ and its corresponding angle configuration $\theta^*:=(\theta_1^*, \ldots, \theta_n^*)\in \mathbb{T}^n$ is given by $\theta_i^*= 0$ if $ i \in V_1^* $ and  $\theta_i^*= \pi $ if $i \in V_2^*$, $i=1, \ldots, n$. Then, for the class of coupling functions $g$ satisfying
\[g \in \mathcal{G^*}=\left \{g\in \mathcal{G} \mid \min _{x \in[0, \pi]} \frac{2}{\pi} \frac{x}{1-g(x)}=1\right\},\]
the global minimum of the energy function $L(\theta ; A, g)$, as defined in \eqref{eq:energy-g}, is attained and only attained by $\theta^*$.

\end{theorem}

\begin{proof}
We prove the statement by contradiction and start by noting that,
for the max-cut configuration $\theta^*$, the number of expected edges obtained from the random rounding procedure is $\mathbb{E} W_{\theta^*}=W_{\mathrm{mc}}$.
For $g \in \mathcal{G^*}$ and a generic configuration $\theta:=(\theta_1, \ldots, \theta_n) \in \mathbb{T}^n$ of the energy function $L(\theta ; A, g)$, the number of expected edges obtained from the random rounding procedure is  
\begin{equation}
    \begin{aligned}
\mathbb{E} W_\theta&=\sum_{1 \leqslant i<j \leqslant n} a_{i j} \frac{\left|\theta_i-\theta_j\right|_{\mathbb{S}^1}}{\pi} =\sum_{1 \leqslant i<j \leqslant n} a_{i j} \frac{2}{\pi} \frac{\left|\theta_i-\theta_j\right|_{\mathbb{S}^1}}{1-g \left(\theta_i-\theta_j\right)} \frac{1-g \left(\theta_i-\theta_j\right)}{2} \\
&\geqslant\left(\min _{x \in [0, \pi]} \frac{2}{\pi} \frac{x}{1-g (x)}\right) \sum_{1 \leqslant i<j \leqslant n} a_{i j} \frac{1-g \left(\theta_i-\theta_j\right)}{2}=\frac{1}{2}(|E|-L(\theta ; A, g)),
\end{aligned}
\end{equation} where $|E|$ denotes the total number of edges.
Suppose there exists $\underline{\theta} \in \mathbb{T}^n$ such that
$L(\underline{\theta} ; A, g)<L\left(\theta^* ; A, g\right)$.
For such a configuration $\underline{\theta}$, the number of expected edges from random rounding procedure would be 
\[\mathbb{E} W_{\underline{\theta}}  \geqslant \frac{1}{2}(|E|-L(\underline{\theta} ; A, g)) 
 > \frac{1}{2}\left(|E|-L\left(\theta^*, A, g\right)\right) 
=W_{\mathrm{mc}},
\]
which leads to a contradiction. Thus we have $\min L(\theta ; A, g)=L\left(\theta^* ; A, g\right)$.
We proceed to prove that the lowest energy configuration is only attained by the max-cut solution. Suppose there exists a non-binarized configuration $\theta' \in \mathbb{T}^n$ with $L\left(\theta' ; A, g\right)=\min L(\theta ; A, g)=L\left(\theta^* ; A, g\right)$.
For such a $\theta'$, $\mathbb{E} W_{\theta'} \geqslant \frac{1}{2}(|E|-L(\theta' ; A, g))=W_{\mathrm{mc}}$.
This leads to a contradiction: $G$ has a unique pair of max-cut solutions,  and therefore, applying the random rounding procedure to any non-binarized configuration will give an expected cut value strictly less than $W_{\mathrm{mc}}$. This is due to the fact that there will always be a sub-optimal binarized configuration with non-zero probability measure.
\end{proof}

\begin{figure}[h]
\label{fig:phase-bistability}
  \centering

    \includegraphics[width=.49\textwidth]{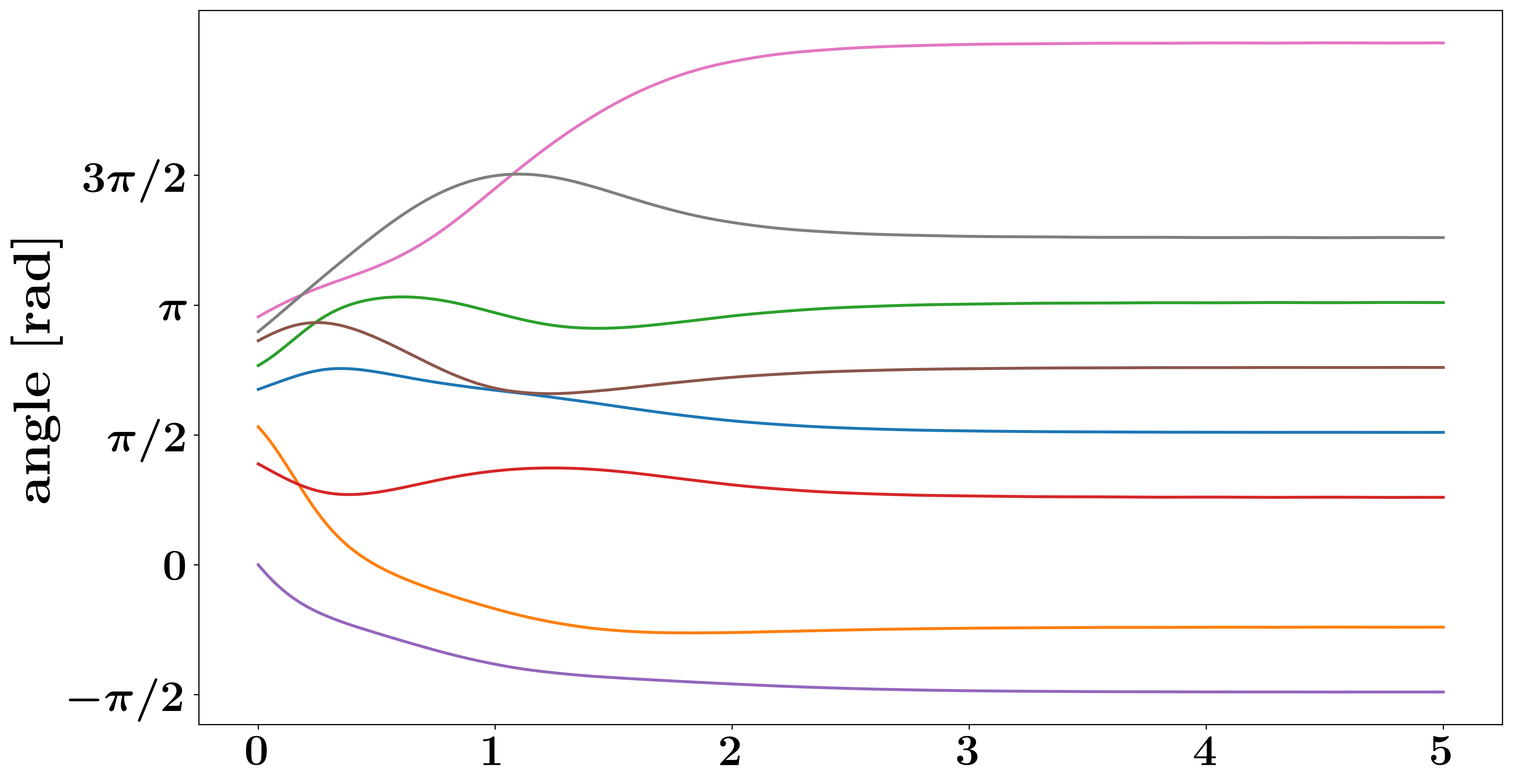}
  \includegraphics[width=.49\textwidth]{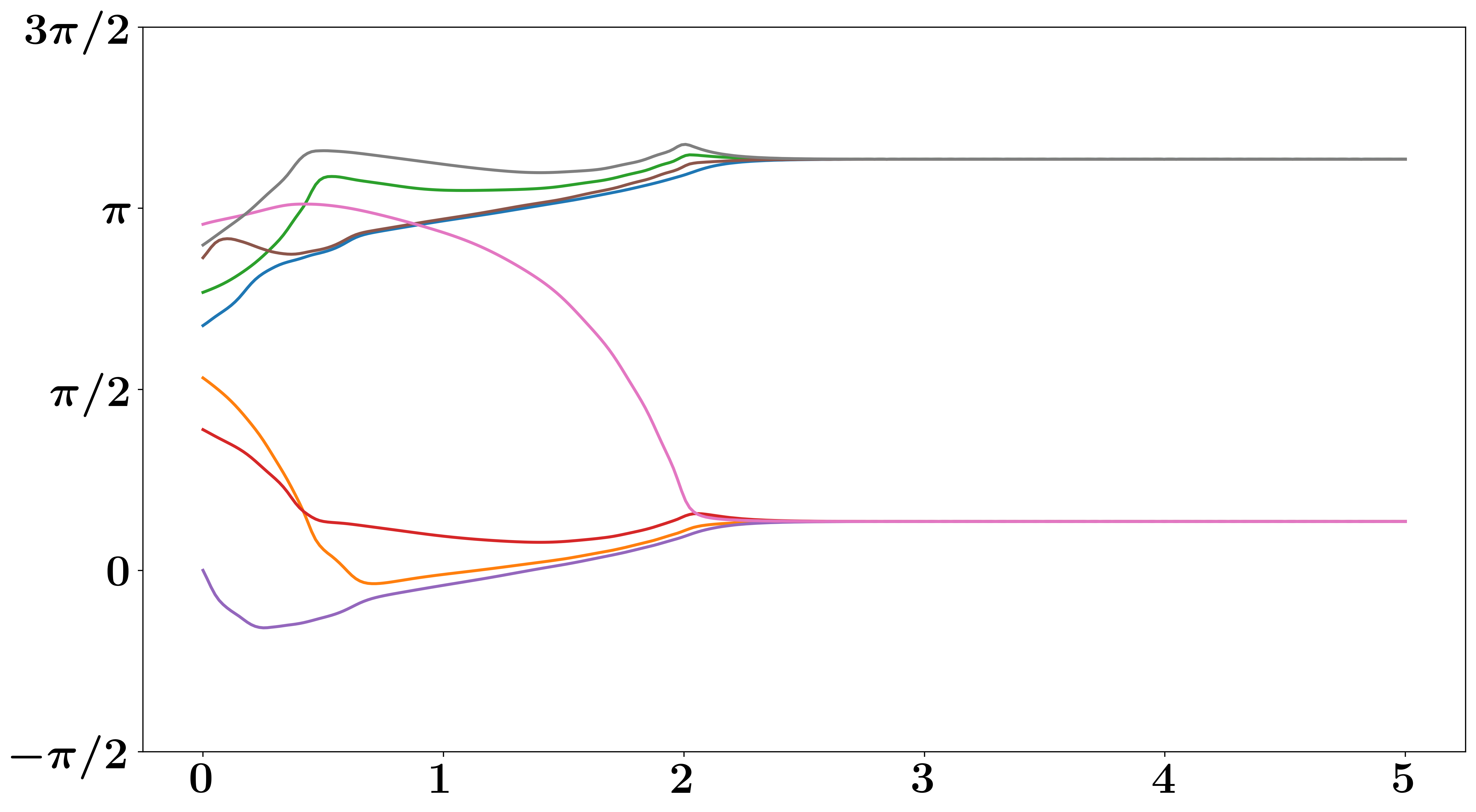}
  \includegraphics[width=.49\textwidth]{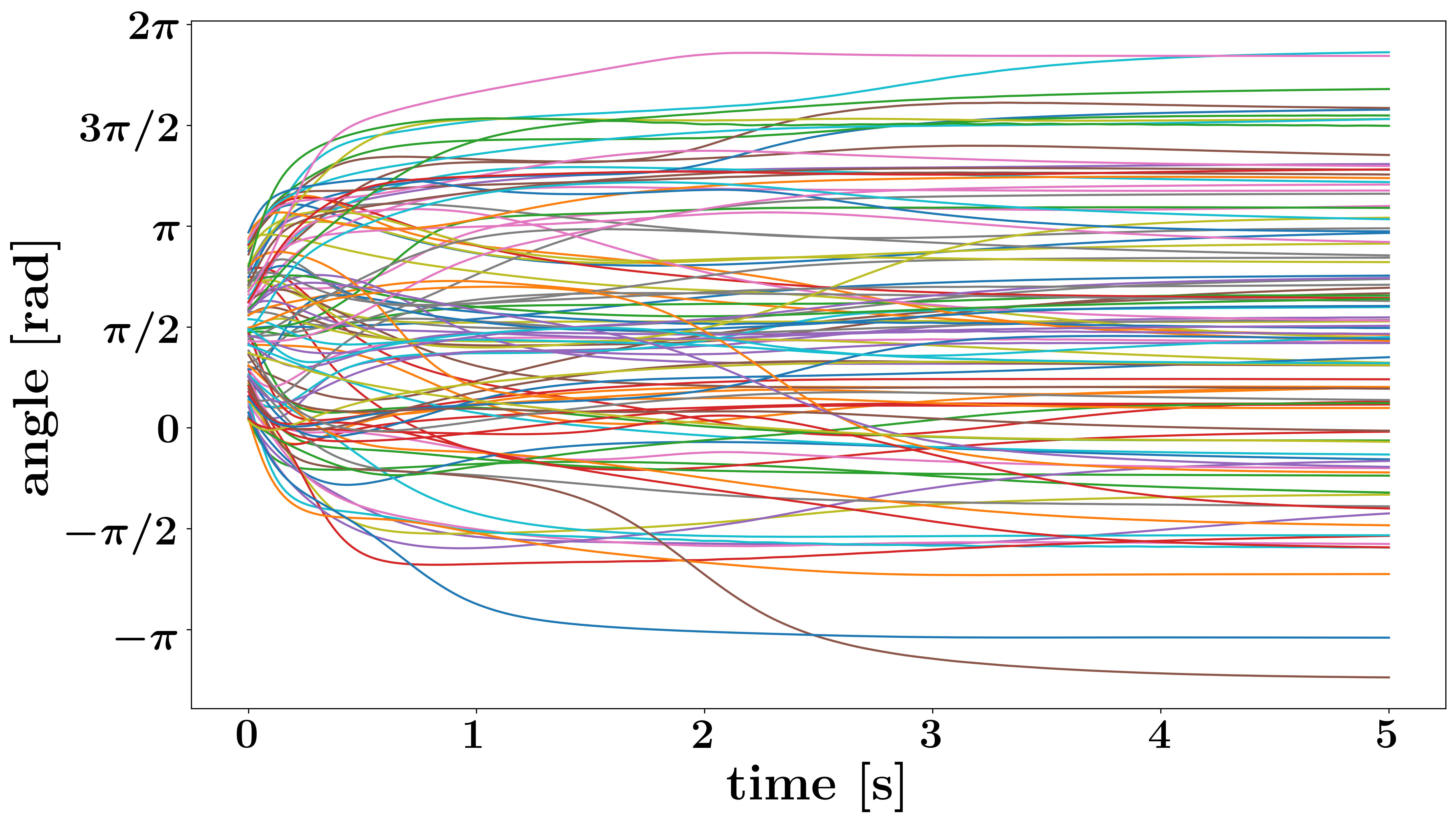}
   \includegraphics[width=.49\textwidth]{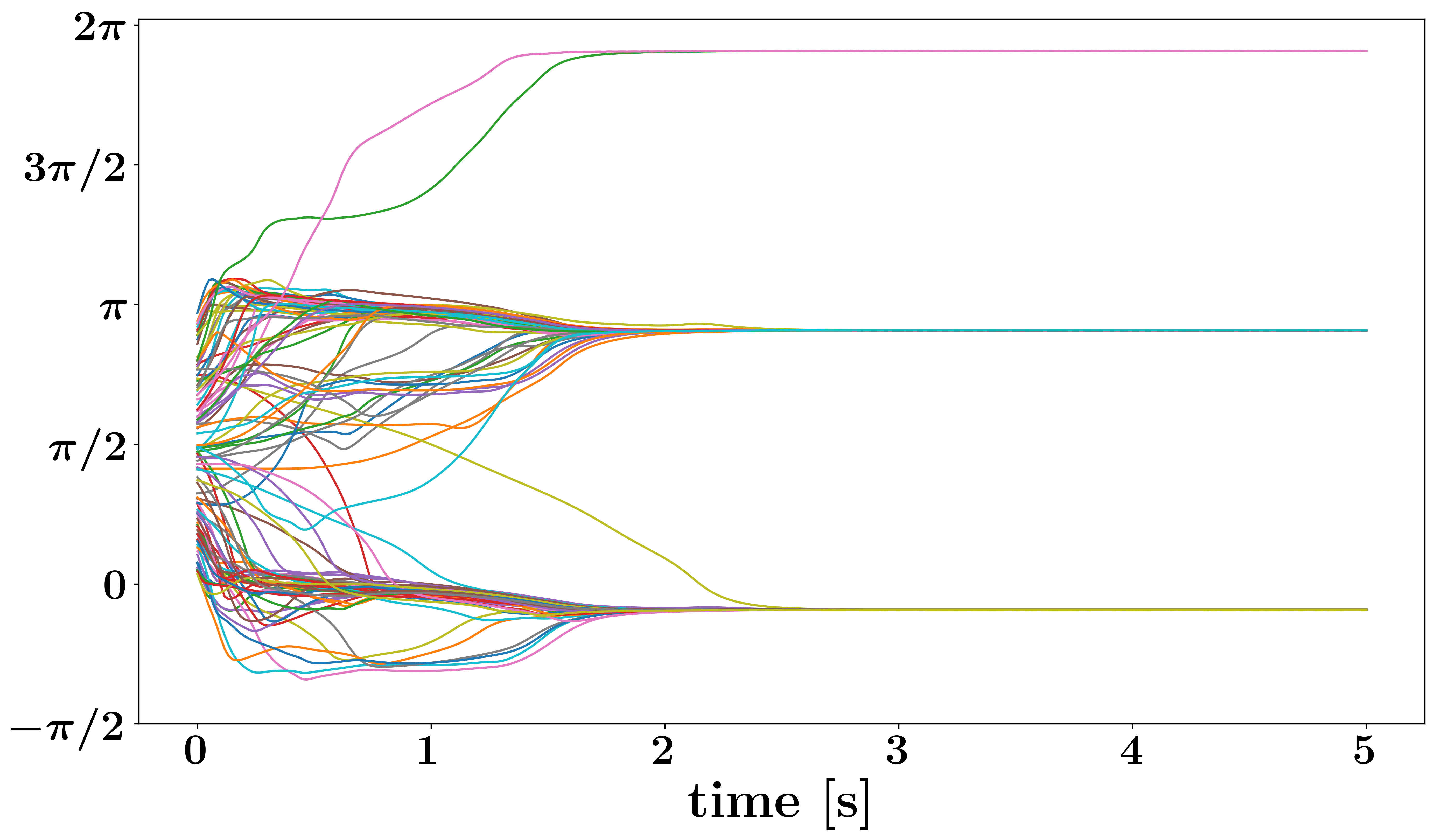}




\caption{The plot shows the evolution of different phase trajectories. The horizontal axis represents time and the vertical axis represents angle (in radians). The first row shows the trajectories resulting from solving max-cut for an $8$-node cubic graph and the second row shows results from an Erd\H{o}s R\'enyi random graph. The left two subplots are from the cosine coupling $g_1$ and the right are from the $10$-term Fourier expansion of $g_2$. In the $8$-node case, the result arising from $g_2$ achieves the global minimum, however for the $100$-node case, only a local minimum is achieved.}
\end{figure}



We conclude the section by highlighting the results of Thm.~\ref{thm:optimal coupling} in numerical simulations. The simulations are based on simulating the gradient flow dynamics of $L(\theta ; A, g)$ using the Fehlberg method (RKF45) with $10^{-3}$ for relative tolerances and $10^{-6}$ for absolute tolerances. We present phase trajectories for the $8$-node cubic graph from \cite{Wang2019OIMOI} and an Erd\H{o}s R\'enyi random graph with $100$ nodes and where each edge is included in the graph with probability $0.06$. We apply the coupling functions $g_1 = \mathrm{cos}(x)$ and $g_2=1-2 x^2 / \pi^2$, whereby we approximate $g_2$ by its $10$-term Fourier expansion. The Fourier expansion ensures differentiability, which is required for simulating the gradient flow dynamics. We note that $g_1$ achieves an approximation ratio of $0.878$, whereas $g_2$ has an approximation ratio of $1$. The numerical results support the conclusion of Thm.~\ref{thm:optimal coupling}, since the $\mathrm{cos}(\cdot)$ coupling leads to non-binarized configurations, whereas the $g_2$ coupling leads to binarized configurations (see Fig.~\ref{fig:phase-bistability}). We note, however, that while in the 8-node configuration, the global minimum is achieved with $g_2$ (and corresponds to a max-cut solution), in the 100-node configuration, only a local minimum is found with $g_2$.

\section{Conclusion}

We presented a dynamical systems perspective on discrete optimization using Ising models and coupled oscillators. We showed that our formulation based on penalty methods is equivalent to the dynamics of oscillator-based Ising machines. In this way, a rigorous rationale for the introduction of sub-harmonic injection locking mechanisms is provided. We then analyzed the advantage of introducing such bi-stability penalties from a random rounding point of view. Furthermore, even in the absence of explicit penalty terms, we characterized the class of generalized coupling functions among oscillators that ensures bi-stability. This guarantees that our dynamics converge to discrete solutions, which is also highlighted with numerical experiments. However, as we observed from the numerical results, bi-stable configuration also occurs in local minima. This reflects a limitation of our proof technique for Thm.~\ref{thm:optimal coupling}, which indicates an interesting future avenue on a constructive proof that includes the consideration for local minima and explains why bi-stability happens. 


\acks{The authors thank the German Research Foundation and the Branco Weiss Fellowship, administered by ETH Zurich, for the support.}

\newpage
\bibliography{l4dc2023-dsco.bib}

\newpage

\appendix






 


\begin{proof}[of Prop.~\ref{thm:I_mu}]
Suppose that we have a configuration $\theta=\left(\theta_1, \ldots, \theta_n\right) \in \mathbb{T}^n$ with small energy. The expected number of edges being cut by a random line can be computed according to \eqref{eq:random-rounding expectation}:
\begin{equation*}
    \mathbb{E} W_\theta=\sum_{i<j} a_{i j} \mathbb{P}\left(\theta_i, \theta_j \text { in different half-spaces}\right)=\sum_{i<j} a_{i j} \frac{\left|\theta_i-\theta_j\right|_{\mathbb{S}^1}}{\pi}.
\end{equation*}
Using \eqref{eq: ratio} and taking into account the fact that all possible pairwise angle differences now have a shrinked range $I_\mu \subseteq[0, \pi]$ due to the penalty, we have 
\begin{equation*}
    \begin{aligned}
\mathbb{E} W_\theta&=\sum_{1 \leqslant i<j \leqslant n} a_{i j}  \cdot \frac{\left|\theta_i-\theta_j\right|_{\mathbb{S}^1}}{\pi} \\ &=\sum_{1 \leqslant i<j \leqslant n} a_{i j} \cdot \frac{2}{\pi} \frac{\left|\theta_i-\theta_j\right|_{\mathbb{S}^1}}{1-\cos \left(\theta_i-\theta_j\right)} \cdot \frac{1-\cos \left(\theta_i-\theta_j\right)}{2} \\
&\geqslant\left(\min _{x \in I_\mu} \frac{2}{\pi} \frac{x}{1-\cos (x)}\right) \sum_{1 \leqslant i<j \leqslant n} a_{i j} \cdot \frac{1-\cos \left(\theta_i-\theta_j\right)}{2}\\
&\geqslant\left(\min _{x \in I_\mu} \frac{2}{\pi} \frac{x}{1-\cos (x)}\right)\left( \sum_{1 \leqslant i<j \leqslant n} a_{i j} \cdot \frac{1-\cos \left(\theta_i-\theta_j\right)}{2}-\frac{\mu}{4} \sum_{i=1}^n \sin ^2 \theta_i\right)\\
&\geqslant\left(\min _{x \in I_\mu} \frac{2}{\pi} \frac{x}{1-\cos (x)}\right)\frac{1}{2}(|E|-L(\theta ; \mu)).
\end{aligned}
\end{equation*}

Suppose the value of max-cut $W_{\mathrm{mc}}$ is attained by the partition $\left(V_1^{*}, V_2^{*}\right)$. The corresponding configuration $\theta^*=\left(\theta_1^*, \ldots, \theta_n^*\right)$ is given by, for $i=1, \ldots, n$,
\[\theta_i^*= \begin{cases}0 & \text { if } i \in V_1^* \\ \pi & \text { if } i \in V_2^*.\end{cases}\]
We now verify that
\[L\left(\theta^* ; \mu\right)=|E|-2 W_{\mathrm{mc}}.\]
Clearly, for the minimum energy configuration $\underline{\theta}\in \underset{\theta\in \mathbb{T}^n}{\arg \min } L(\theta ; \mu)$, we have
\[L\left(\underline{\theta}; \mu\right) \leqslant L\left(\theta^* ; \mu\right)=|E|-2 W_{\mathrm{mc}}.\]
Therefore the expected number of edges $\mathbb{E} W_{\underline{\theta}}$ for a random rounding partition of the minimal energy configuration $\underline{\theta}$ satisfies 
\begin{equation*}
\begin{aligned}
\mathbb{E} W_{\underline{\theta}}
&\geqslant\left(\min _{x \in I_\mu} \frac{2}{\pi} \frac{x}{1-\cos (x)}\right) \frac{1}{2}(|E|-L(\underline{\theta} ; \mu))\\
&\geqslant\left(\min _{x \in I_\mu} \frac{2}{\pi} \frac{x}{1-\cos (x)}\right)  W_{\mathrm{mc}}.
\end{aligned}
\end{equation*}
This completes the proof.
\end{proof}

\end{document}